\documentclass[reqno]{amsart}
\usepackage{amsmath, amssymb, amsthm, epsfig}
\usepackage{hyperref, latexsym}
\usepackage{url}
\usepackage[mathscr]{euscript}

\usepackage{color}
\usepackage{fullpage} 
\usepackage{setspace}
\usepackage{xcolor}

\def\today{\ifcase\month\or
  January\or February\or March\or April\or May\or June\or=
  July\or August\or September\or October\or November\or December\fi
  \space\number\day, \number\year}

 \newtheorem{theorem}{Theorem}
  
 \newtheorem{lemma}[theorem]{Lemma}
 \newtheorem{proposition}[theorem]{Proposition}
 
 \theoremstyle{definition}

 \theoremstyle{remark}

    \renewcommand{\d}{\text{\rm d}}

\newcommand{\var}{{\rm Var\,}}
\newcommand{\intav}[1]{\mathchoice {\mathop{\vrule width 6pt height 3 pt depth  -2.5pt
\kern -8pt \intop}\nolimits_{\kern -6pt#1}} {\mathop{\vrule width
5pt height 3  pt depth -2.6pt \kern -6pt \intop}\nolimits_{#1}}
{\mathop{\vrule width 5pt height 3 pt depth -2.6pt \kern -6pt
\intop}\nolimits_{#1}} {\mathop{\vrule width 5pt height 3 pt depth
-2.6pt \kern -6pt \intop}\nolimits_{#1}}}

\begin{document}

\title[On the continuity of maximal operators of convolution type at the derivative level]{On the continuity of maximal operators of convolution type at the derivative level}
\author[Gonz\'{a}lez-Riquelme]{Cristian Gonz\'{a}lez-Riquelme}
\address{IMPA - Instituto de Matem\'{a}tica Pura e Aplicada\\
Rio de Janeiro - RJ, Brazil, 22460-320.}
\email{cristian@impa.br}

\allowdisplaybreaks
\numberwithin{equation}{section}

\maketitle

\begin{abstract}
In this paper we study a question related to the continuity of maximal operators of convolution type acting on $W^{1,1}(\mathbb{R})$. We prove that the map $u\mapsto (u^{*})'$ is continuous from $W^{1,1}(\mathbb{R})$ to $L^{1}(\mathbb{R})$, where $u^{*}$ is the maximal function associated to the Poisson kernel, the Heat kernel or a family of kernels related to the fractional Laplacian. This is the first result of this type for a centered maximal operator.   

\end{abstract}
\section{Introduction}
\subsection{A brief historical background} Maximal operators are central objects in analysis. The regularity theory for these operators was initiated by Kinnunen \cite{Kinnunen1997}, who proved that \begin{align}\label{map}f\mapsto Mf\end{align} is bounded from $W^{1,p}(\mathbb{R}^d)$ to itself, where $M$ is the centered Hardy-Littlewood maximal operator and $p>1$. The method presented in that work also allows one to establish the same result for the uncentered Hardy-Littlewood maximal operator $\widetilde{M}$. This continues to hold true at the derivative level when $p=1$
and $n = 1$, i.e. the map $f\mapsto (Mf)'$ is bounded from $W^{1,1}(\mathbb{R})$ to $L^{1}(\mathbb{R}).$ This was proved by Kurka in the work \cite{Kurka2010}, while the same result for $\widetilde{M}$ was established by Tanaka \cite{Tanaka2002} and refined by Aldaz and P\'{e}rez-L\'{a}zaro \cite{AP2007}. The same result for higher dimensional radial functions was proved by Luiro \cite{Luiro2018} and for characteristic higher dimensional functions was obtained by Weigt \cite{weigtchar}, both in the uncentered case. Extending this kind of result to different settings and operators  has been an important line of research the last decades. For general information about this area there is an interesting survey \cite{CarneiroSurvey} by Carneiro.

The continuity for this type of operators has been also extensively studied. Luiro \cite{Luiro2007} proved that \eqref{map} is continuous from $W^{1,p}(\mathbb{R}^d)$ to itself when $p>1$. The methods developed in that work can be adapted to several other maximal operators in the range $p>1.$ The endpoint case $p=1$ is significantly more involved. For the uncentered Hardy-Littlewood maximal operator the continuity of the map
\begin{align}\label{map2}
f\mapsto \left(\widetilde{M}f\right)'
\end{align}
from $W^{1,1}(\mathbb{R})$ to $L^{1}(\mathbb{R})$ was proved by Carneiro, Madrid and Pierce in \cite{CMP2017}. This was later generalized in \cite{grk} and \cite{CGRM} to the $BV(\mathbb{R})$ case and to the higher dimensional radial case, respectively. In the fractional variant the general case was obtained in \cite{beltran2021continuity}, after previous developments made in \cite{BM2019,Madrid2017}. However, the classical centered setting is out of the scope of the previous methods. In this manuscript we aim to shed light in this kind of question. 
\subsection{Maximal operators of convolution type}
Let $\phi:\mathbb{R}^d \mapsto \mathbb{R}_{\ge 0}$ be a radially non increasing function with $$\int_{\mathbb{R}^d}\phi(x)\,\d x=1.$$ We define, as usual, \begin{align}\label{approximationofidentity}
\phi_{t}(x):=\frac{1}{t^d}\phi\left(\frac{x}{t}\right).
\end{align}
For every $f\in L^{1}(\mathbb{R}^d)$ we have $\underset{t\to 0}{\lim}\ |f|*\phi_{t}(\cdot)(x)=|f|(x)$ for a.e. $x\in \mathbb{R}^d$. Then, given an initial datum $u:\mathbb{R}^d\to \mathbb{R}$ we define the extension $\tilde{u}:\mathbb{R}^d\times (0,\infty)\to \mathbb{R}$ as
$$\tilde{u}(x,t)=|u|*\phi_t(x).$$
The maximal operator associated to the kernel $\phi$ is defined as
$$u^{*}(x)=\sup_{t>0}\tilde{u}(x,t),$$ 
where we omit the dependence to $\phi$ as it is clear from the context.
This notion recovers the classical Hardy-Littewood maximal operator by choosing $\phi=\frac{\chi_{B(0,1)}}{|B(0,1)|}.$
The following kernels are the main objects for our study:
\begin{align*}
\varphi_{1}(x) &= \frac{\Gamma \left(\frac{d+1}{2}\right)}{\pi^{(d+1)/2}}\ \frac{1}{(|x|^2 + 1)^{(d+1)/2}} \ \ \ \ \ \ \ {\rm (Poisson \ kernel)}\\
\varphi_{2}(x) & = \frac{1}{(4 \pi )^{d/2}}\ e^{-|x|^2}  \ \ \ \ \ \ \ \ \ \ \ \ \ \ \ \ \ \ \ \ \ \ \ {\rm (Heat \ kernel)}\\
\varphi_{3}^{\alpha}(x)&=C_{d}^{\alpha}\frac{1}{(|x|^2 + 1)^{(d+1-\alpha)/2}} \ \ \ \ \ \ \ \ {\rm (Fractional\ Poisson\ kernel)} ,
\end{align*}
where $0<\alpha<1$ and $C_{d}^{\alpha}$ is such that $\|\varphi_{3}^{\alpha}\|_1=1.$ We notice that $\varphi_{3}^{\alpha}$ is a generalization of $\varphi_1$. We treat them as different objects since some results for $\varphi_1$ were previously obtained.  The kernel $\varphi_3^{\alpha}$ is relevant in the work of Caffarelli and Silvestre \cite{doi:10.1080/03605300600987306}, where its relation to the fractional Laplacian is investigated.
The study of regularity properties for maximal operators associated to smooth kernels was initiated by Carneiro and Svaiter in \cite{CS2013}. There, it was proved the endpoint boundedness of the map 
\begin{align}\label{convolution}
    u\mapsto (u^{*})'
\end{align}
from $W^{1,1}(\mathbb{R})$ to $L^{1}(\mathbb{R})$ where $\phi$ is $\varphi_{1}$ or $\varphi_{2}$. One key step in the proof of these results was to prove that in the set $\{x\in \mathbb{R};u^*(x)>u(x)\}$ the function $u^{*}$ is subharmonic for such kernels. This regular behavior has proven useful also in higher dimensions, where it was used to prove that in fact these maximal operators reduce the $L^{2}$-norm of the gradient \cite{CS2013}. These results were later generalized in \cite{bortz2019sobolev,CFS2015}. In \cite{CGR} the boundedness of \eqref{convolution} in the higher dimensional radial case was established for both $\varphi_{1}$ and $\varphi_{2}$.

In our main result we make use of the essential {\it subharmonicity property} to conclude the one-dimensional continuity of these maximal operators at the derivative level, solving a question posed by Carneiro\footnote{Personal communication.}.  
\begin{theorem}\label{1}
Let $\phi\in \{\varphi_{1},\varphi_{2},\varphi_{3}^{\alpha}\}$. Then the map $$u\mapsto (u^{*})'$$ is continuous from $W^{1,1}(\mathbb{R})$ to $L^{1}(\mathbb{R}).$
\end{theorem}
In the case of $\varphi_{3}^{\alpha}$ we first have to prove that the aforementioned map is well defined and bounded. This is obtained by similar methods as the ones developed in \cite{CFS2015,CS2013}. This is explained in Section 2.

The methods developed in the aforementioned works \cite{CGRM,CMP2017,grk} are not enough to conclude Theorem \ref{1}. In those works it is relevant that the maximal operators considered there have the flatness property; that is, the maximal functions have zero derivative a.e. at the points where they coincide with the original functions. This property does not typically hold when dealing with centered maximal operators, so a new approach is required in this case. In order to overcome this difficulty, our strategy is strongly tied with the {\it subharmonicity property} that these kernels satisfy. We use this property in order to obtain a {\it local boundedness} that is stable under linear perturbations. This allows us to discretize some important aspects of the proof. Complementing this with some previous methods developed in \cite{grk} we obtain our result. These new tools are explained in Section 3.     
\section{Preliminaries}
Here we develop the preliminaries for the proof of our theorem. Given $u\in W^{1,1}(\mathbb{R})$ we write its disconnecting set as $$D:=\{x\in \mathbb{R}; u^{*}(x)>u(x)\}.$$   
We say that $\phi\in L^{1}(\mathbb{R})$ has the {\it subharmonicity property} when for any $u\in W^{1,1}(\mathbb{R})$ the associated maximal operator $u^{*}$ is subharmonic in $D$. We notice that, given that we are in the one-dimensional setting, this property implies that $u^{*}$ is convex in $D$.   By \cite[Lemmas 8 and 12]{CS2013} we know that property holds for both $\varphi_1$ and $\varphi_2.$ In the next proposition we establish the same for $\varphi_3^{\alpha}.$
\begin{proposition}
For $\phi=\varphi_3^{\alpha}\in L^{1}(\mathbb{R}^{d}),$ $\alpha\in (0,1),$ we have that $u^{*}$ is continuous in $\mathbb{R}^d$ and subharmonic in the set $\{x\in \mathbb{R}^{d};u^{*}(x)>u(x)\}$ for any $u\in W^{1,1}(\mathbb{R}^{d})\cap C(\mathbb{R}^d).$ Moreover, the map $u\mapsto (u^{*})'$ is well defined and bounded from $W^{1,1}(\mathbb{R})$ to $L^{1}(\mathbb{R}).$
\end{proposition}
\begin{proof}
Following \cite[Lemma 7(i)]{CS2013} we can conclude that $u^{*}$ is continuous for $u\in W^{1,1}(\mathbb{R}^d)\cap C(\mathbb{R}^d).$ Therefore, following the proof of \cite[Theorem 2(ii)]{CS2013}, we need to prove the fact that $u^{*}$ is subharmonic in the set $\{u^{*}>u\}$ to conclude the last assertion of our proposition. In order to conclude this subharmonicity, let us observe that, according to \cite[Section 2.4]{doi:10.1080/03605300600987306}, the function $\tilde{u}(\cdot,t):=u*(\varphi_{3}^{\alpha})_{t}(\cdot)$ solves the Cauchy problem \begin{equation*}
\begin{split}
&\bigtriangleup_{x}\tilde{u}+\frac{\alpha}{t}\tilde{u}_t+\tilde{u}_{tt}=0 \ \ \ \ \ \ \ {\rm{for}}\ {(x,t)\in \mathbb{R}^d\times (0,\infty)}\\
    &\ \ \ \ \ \ \tilde{u}(x,0)=u(x) \ \ \ \ \ \ \ \ \ \ \ \ \ \ {\rm{for}}\ {x\in \mathbb{R}^d}.
\end{split}
\end{equation*}
Therefore, by combining \cite[Theorem 3.1]{Gilbarg2001} and the remark thereafter with the proof of \cite[Lemma 7]{CFS2015}, we just need to prove the following: for any compact ball $B_{r}(x_0)$ and $\varepsilon>0$, there exists $t_{\varepsilon}$ big enough such that for any $z\in B_{r}(x_0)$ we have $\tilde{u}(z,t)<\varepsilon$ for any $t>t_{\varepsilon}.$ This claim follows from 
$$|\tilde{u}(z,t)|\le \|(\varphi_3^{\alpha})_{t}\|_{\infty}\|u\|_{1},$$
and the fact that $\|(\varphi_3^{\alpha})_{t}\|_{\infty}\to 0$ when $t\to \infty.$
\end{proof}
Let us recall the following result.
\begin{lemma}[{\cite[Lemma 14]{CMP2017}}]
Let $u\in W^{1,1}(\mathbb{R})$ and $\{u_j\}_{j\ge 1}\subset W^{1,1}(\mathbb{R})$ such that $\|u_j-u\|_{1,1}\to 0$. Then $\||u_j|-|u|\|_{1,1}\to 0.$
\end{lemma}
The previous result allow us to always assume that the $u_j$ and $u$ are nonnegative, a simplification that we adopt henceforth.
Now we prove a general statement about the uniform convergence of maximal functions. 
\begin{proposition}\label{uniformconvergence}
Let $u_j\to u$ in $W^{1,1}(\mathbb{R}).$ Then $$u_j^{*}\to u^{*}$$ uniformly.
\end{proposition}
\begin{proof}
Let $x\in \mathbb{R},$ and let $t_x, t_{x,j}\ge 0$ such that $\tilde{u}(x,t_x)=u^{*}(x)$ and $\tilde{u}_{j}(x,t_{x,j})=u_j^{*}(x)$, where we use the notation $\tilde{u}(x,0):=u(x)$. Then \begin{align*}
|u_{j}^{*}(x)-u^{*}(x)|&=\max\{u_{j}^{*}(x)-u^{*}(x),u^{*}(x)-u_{j}^{*}(x)\}\\
&\le \max\{\tilde{u}_{j}(x,t_{x,j})-\tilde{u}(x,t_{x,j}),\tilde{u}(x,t_x)-\tilde{u}_j(x,t_{x})\}\\
&\le \|u_{j}-u\|_{\infty}\\
&\le \|u_{j}-u\|_{1,1}.
\end{align*}
\end{proof}
In the following we assume $\phi\in \{\varphi_1,\varphi_2,\varphi_3^{\alpha}\}$ and that $d=1$. Recall that in that case $u^{*}$ is weakly differentiable and continuous. In the next lemma we reduce our analysis to a bounded interval.
\begin{lemma}\label{compactness}
If $u_j\to u$ in $W^{1,1}(\mathbb{R})$, for every $\varepsilon>0$ there exist $R>0$ and there exists $j$ big enough such that we have
$$\int_{[-R,R]^c}\left|(u_{j}^{*})'\right|+\left|(u^{*})'\right|<\varepsilon$$
\end{lemma}
\begin{proof}
We prove that for any function $w\in W^{1,1}(\mathbb{R})$ and $R>0$ we have $$\int_{[R,\infty)}\left|(w^{*})'\right|\le |w^{*}(R)-w(R)|+\int_{[R,\infty)}\left|w'\right|,$$ the other required estimate follows by symmetry.
If we write $\{x\in(R,\infty); w^{*}(x)>w(x)\}=\bigcup_{i=1}^{\infty}(a_i,b_i)$ we have $w^{*}(a_i)=w(a_i)$ and $w^{*}(b_i)=w(b_i)$ unless $a_i=R.$ If $a_i\neq R$ we have $\int_{(a_i,b_i)}|(u^{*})'|\le \int_{(a_i,b_i)}|(u)'|$ by the {\it subharmonicity property}. By the same property we have that, if $a_i=R$ and $w^{*}$ attains its minimum for the interval $(a_i,b_i)$ at the point $c_i$, we have  (with possibly $b_i=\infty$)
\begin{align*}
\int_{(R,b_i)}|(w^{*})'|&=w^{*}(R)-2w^{*}(c_i)+w^{*}(b_i)\\
&\le |w^{*}(R)-w(R)|+w(R)-2w(c_i)+w(b_i)\le |w^{*}(R)-w(R)|+\int_{(R,b_i)}|w'|,
\end{align*}
from where we conclude our claim. Now, in order to conclude our lemma we take  $R$ such that $\int_{[R,R]^c}|w'|\le \frac{\varepsilon}{4}$, $w^{*}(R)-w(R)+w^{*}(-R)-w(-R)<\frac{\varepsilon}{4}$ and $j$ such that $\|w_j'-w'\|_{1}<\frac{\varepsilon}{4}$ and $w_j^{*}(R)-w_j(R)+w_j^{*}(-R)-w_j(-R)<\frac{\varepsilon}{4},$ where in this last choosing we use Proposition \ref{uniformconvergence}. 
\end{proof}
Another important ingredient in our strategy is presented in the next proposition, which is inspired in \cite[Proposition 8]{grk}. For a partition $\mathcal{P}:=\{a_1<\dots<a_m\}$ and $w:\mathbb{R}\to \mathbb{R}$ we define 
$$\var(w,\mathcal{P}):=\sum_{i=1}^{m-1}|w(a_{i+1})-w(a_{i})|.$$
\begin{proposition}\label{grk}
Let $u_j\to u\in W^{1,1}(\mathbb{R}).$
Then $$\|(u_j^{*})'\|_{1}\to \|(u^{*})'\|_{1}$$ 
\end{proposition}
\begin{proof}
 By Lemma \ref{compactness} is enough to prove that, for any $(a,b)$, we have $$\underset{j\to \infty}{\lim} \int_{[a,b]}\left|(u_{j}^{*})'\right|=\int_{[a,b]}\left|(u^{*})'\right|.$$ Since fo any $w\in W^{1,1}(\mathbb{R})$ we have $$\int_{[a,b]}|w'|=\underset{\mathcal{P}\subset [a,b]}{\sup}\var(w,\mathcal{P}),$$ by Fatou's lemma we obtain $$\underset{j\to \infty}{\liminf} \int_{[a,b]}\left|(u_{j}^{*})'\right|\ge \int_{[a,b]}\left|(u^{*})'\right|.$$ Now, given $\varepsilon>0,$ we prove that $$\underset{j\to \infty}{\limsup} \int_{\mathbb{R}}\left|(u_{j}^{*})'\right|\le \int_{[a,b]}\left|(u^{*})'\right|+3\varepsilon.$$ 
 Let us take a partition $\mathcal{P}=\{a=a_0<a_1\dots<a_{K}=b\}$ such that 
\begin{align*}
\var\left(u,\mathcal{P}\right)>\int_{[a,b]}\left|u'\right|-\varepsilon
\end{align*}
and
\begin{align*}
\var\left(u^{*},\mathcal{P}\right)>\int_{[a,b]}\left|(u^{*})'\right|-\varepsilon.
\end{align*}
By uniform convergence we have 
\begin{align}\label{epsilon1}
\var\left(u_j,\mathcal{P}\right)>\int_{[a,b]}\left|(u_{j})'\right|-2\varepsilon 
\end{align}
and 
\begin{align}\label{epsilon2}
\var\left(u_j^{*},\mathcal{P}\right)>\int_{[a,b]}\left|(u_{j}^{*})'\right|-2\varepsilon
\end{align}
for $j$ big enough. Now, let us consider $\widetilde{\mathcal{P}}=\widetilde{\mathcal{P}}(j)\supset \mathcal{P}$ with $\widetilde{\mathcal{P}}\subset [a,b]$ such that
\begin{align*}
\var\left(u_{j}^{*},\widetilde{\mathcal{P}}\right)>\int_{[a,b]}\left|(u_j^{*})'\right|-\varepsilon.
\end{align*}
Without loss of generality we can assume that $\widetilde{\mathcal{P}}$ is such that $[a_i,a_{i+1}]\cap \widetilde{\mathcal{P}}=\{a_{i-1}=a_{i,0}<\dots<a_{i,n_i}=a_i\}$ satisfies that $\text{sign}(u^{*}_{i,k}-u^{*}_{i,k+1})=-\text{sign}(u^{*}_{i,k+1}-u^{*}_{i,k+2})$ for every $k=0,\dots,n_{i}-2$. For each such $i$ we denote $\widetilde{\mathcal{P}}_{i}=\{a_{i,1},\dots,a_{i,n_{i}-1}\}$ and claim that it is possible to find another partition $\widetilde{\mathcal{P}}_{i}^{*}=\{a_{i,1}^{*},\dots,a_{i,n_{i}-1}^{*}\}\subset (a_{i-1},a_i)$ such that 
\begin{align}\label{localization}
\var\left(u_{j},\widetilde{\mathcal{P}}_{i}^{*}\right)-\var\left(u_{j},\{a_{i,1}^{*},a_{i,n_i-1}^{*}\}\right)=\var\left(u_{j}^{*},\widetilde{\mathcal{P}}_{i}\right)-\var\left(u_{j}^{*},\{a_{i,1}.a_{i,n_{i}-1}\}\right) 
\end{align}
For $n_i\le 2$ it follows by convention. For $n_i\ge 3$, by the {\it subharmonicity property} if $k\in \{i,\dots,n_{i}-1\}$ is such that $u_{j}^{*}(a_{i,k})>\max \{u_{j}^{*}(a_{i,k-1}),u_{j}^{*}(a_{i,k+1})\},$ there exists $y\in (a_{i,k-1},a_{i,k+1})$ such that $u_{j}(y)=u_{j}^{*}(a_{i,k})$. We choose $a_{i,k}^{*}=y$ in that case. Now, if $u_{j}^{*}(a_{i,k})<\min\{u_{j}^{*}(a_{i,k+1}^{*}),u_{j}^{*}(a_{i,k-1}^{*})\}$ (where $a_{i,0}^{*}=a_{i,0}$ and $a_{i,n_i}^{*}=a_{i,n_i}$) and $k<n_{i}-1$, since $u_{j}(a_{i,k})\le u_{j}^{*}(a_{i,k})<u_{j}(a_{i,k+1}^{*})$ by continuity there exists $y\in (a_{i,k},a_{i,k+1})$ such that $u_j(y)=u_{j}^{*}(a_{i,k})$. We choose $y=a_{i,k}^{*}$. The case $k=n_{i}-1$ is done analogously, but instead choosing $y\in (a_{i,n_{i}-2},a_{i,n_{i}-1})$ with the same property. From here \eqref{localization} follows.
Now, we apply \eqref{localization} in order to obtain the following inequality
\begin{align*}
\int_{[a,b]}\left|(u_j)'\right|-\var(u_j,\mathcal{P})&\ge 
\var\Big(u_j, \mathcal{P} \cup \bigcup_{i=1}^K \widetilde{\mathcal{P}}_{i}^{*}\Big) - \var\Big(u_j, \mathcal{P} \cup \bigcup_{i=1}^K \{a_{i,1}^{*},a_{i,n_i-1}^{*}\} \Big)
\\
&=\sum_{i=1}^{K}\var\Big(u_j,\widetilde{\mathcal{P}}_{i}^{*}\Big)-\var\Big(u_j,\{a_{i,1}^{*},a_{i,n_{i}-1}^{*}\}\Big)\\
&\ge \sum_{i=1}^{K} \var\left(u_j^{*},\widetilde{\mathcal{P}}_{i}\right)-\var\left(u_j^{*},\{a_{i,1},a_{i,n_i-1}\}\right)\\
&\ge \var\Big(u_{j}^{*},\widetilde{\mathcal{P}}\Big)-\var\Big(u_{j}^{*},\widetilde{\widetilde{\mathcal{P}}}\Big),
\end{align*}
where $\widetilde{\widetilde{\mathcal{P}}}:=\left\{a_{i,k};i\in \{1,\dots, K\},k\in \{0,1,n_i-1,n_i\}\right\}.$ Notice that $|\widetilde{\widetilde{\mathcal{P}}}|\le 3K+1,$ therefore:
$$\var\left(u_j^{*},\widetilde{\widetilde{\mathcal{P}}}\right)<\var\left(u^{*},\widetilde{\widetilde{\mathcal{P}}}\right)+12 K \|u_j-u\|_{\infty}<\int_{(a,b)}|(u^{*})'|+\varepsilon$$
for $j$ big enough. Combining these estimates with \eqref{epsilon2}, we get 
$$\int_{[a,b]}\left|(u_j)'\right|-\var\left(u_j,\mathcal{P}\right)>\int_{[a,b]}\left|(u_{j}^{*})'\right|-\int_{[a,b]}\left|(u^{*})'\right|-\varepsilon.$$
Then, we have by \eqref{epsilon1} that $$\int_{[a,b]}|(u_{j}^{*})'|\le \int_{[a,b]}|(u^{*})'|+3\varepsilon,$$
from where we conclude.
\end{proof}
Now we state a classical property about convergence of convex functions.
\begin{lemma}\label{convexityconvergence}
Let $\{w_j\}_{j\in \mathbb{N}}$ and $(l_{j},r_j)\subset \mathbb{R}$ with $w_j:\mathbb{R}\to \mathbb{R}$ such that $w_j$ is convex in $(l_{j},r_{j})$ for each $j\in \mathbb{N}$. Assume that $\underset{j\to \infty}{\lim} l_j=l$ and $\underset{j\to \infty}{\lim} r_j=r$ and that $w_{j}\to w$ uniformly. Then $w$ is convex in $(l,r)$ and $$\underset{j\to \infty}{\lim} w_{j}'(x)=w'(x),$$ for a.e. $x\in (l,r).$   
\end{lemma}
\begin{proof}
For $a,b\in (l,r),$ then for $j$ big enough $a,b\in [l_j,r_j]$ and therefore $w_j(\frac{a+b}{2})\le \frac{w_j(a)+w_j(b)}{2}$. Then by the pointwise convergence we have $w(\frac{a+b}{2})\le \frac{w(a)+w(b)}{2}$ from where the convexity follows. Then, the absolutely continuity of $w$ in $(l,r)$ follows. The last claim follows as in \cite[Theorem 25.7]{rockafellar-1970a}.
\end{proof}
The last preliminary lemma is the following.
\begin{lemma}\label{pointwisederivative}
Let $\phi\in \{\varphi_1,\varphi_2,\varphi_3^{\alpha}\}$. If $u_j\to u$ in $W^{1,1}(\mathbb{R})$, then $$(u_{j}^{*})'(x)\to (u^{*})'(x)$$ for a.e. $x\in D.$
\end{lemma}
\begin{proof}
Follows by an adaptation of  \cite[Lemmas 5 and 13]{CMP2017}. A simpler proof using the {\it subharmonicity property} follows by Lemma \ref{convexityconvergence} above.
\end{proof}
\section{Novel tools}

In this section we develop new additional tools to address the continuity problem.
Let us take $\varepsilon>0$, consider $v_{\varepsilon}=\displaystyle \sum_{i=1}^{N}\alpha_{i}\chi_{(a_i,a_{i+1})}$, such that $\|u'-v_{\varepsilon}\|_{1}<\varepsilon$. That is, we approximate the derivative of our limit function by a simple function. We define, as usual $D_{j}:=\{x\in \mathbb{R}; u_j^{*}(x)>u_j(x)\}$. We know that $D_j$ is an open set, we write it as an union of intervals in a convenient way that depends on our approximation $v_{\varepsilon}.$ That is $$D_{j}=D_{j}^{1}\cup \bigcup_{i=0}^{N+1}D_{j}^{2,i},$$
where $D_{j}^{1}$ is the union of the intervals contained in $D_j$ that contain at least one element of the set $\{a_{1},\dots,a_{N+1}\}$. The sets $D_{j}^{2,i}$, for $i=1,\dots,N$ are the union of intervals contained in $(a_{i},a_{i+1}),$ and $D_{j}^{2,0}$ and $D_{j}^{2,N+1}$ are the union of the intervals contained in $D_{j}$ that are contained in $(-\infty,a_{1})$ and $(a_{N+1},\infty)$, respectively. We write $$D_{j}^{1}=\bigcup_{r=1}^{N+1}(c_{r}(j),d_{r}(j)),$$ where $(c_r(j),d_r(j))\ni a_r$ (possibly some intervals are empty or the same) and $$D_{j}^{2,i}=\bigcup_{k=1}^{\infty}(c_{k}^{i}(j),d_{k}^{i}(j)).$$ 
The heart of our proof is the following lemma, where we prove that in the sets $D_{j}^{2,i}$ the function $u_{j}^{*}$ is close to $u_j$ at the derivative level. In the proof the {\it subharmonicity property} plays a major role. 
\begin{lemma}\label{boundingbiggroups}
If $\|u'-u_j'\|_{1}<\varepsilon$ we have that  $$\int_{\bigcup_{i=0}^{N+1}D_{j}^{2,i}}\left|(u_{j}^{*})'-u_j'\right|<4\varepsilon.$$
\end{lemma}
\begin{proof}
Let us define $a_0:=-\infty$, $a_{N+2}:=\infty$ and $\alpha_0:=0=:\alpha_{N+1}$.
Let us see that, for $i=0,\dots,N+1$, $$\int_{\cup_{k=1}^{\infty}(c_k^{i}(j),d_{k}^{i}(j))}\left|(u_{j}^{*})'-u_j'\right|<2\int_{(a_i,a_{i+1})}\left|u_{j}'-\alpha_i\right|,$$ from where the result follows since $$\sum_{i=0}^{N+1}\int_{(a_i,a_{i+1})}\left|u_{j}'-\alpha_i\right|< \varepsilon+\sum_{i=0}^{N+1}\int_{(a_i,a_{i+1})}\left|u'-\alpha_i\right|<2\varepsilon.$$Indeed, consider $L_{i}:(a_i,a_{i+1})\to \mathbb{R}$ a line with $L_{i}'(x)=\alpha_i$ for all $x$ and $i=0,\dots,N+1$. Then we observe that \begin{align*}
\int_{(c_{k}^{i}(j),d_{k}^{i}(j))}\left|(u_{j}^{*})'-u_j'\right|&=\int_{(c_{k}^{i}(j),d_{k}^{i}(j))}\left|(u_{j}^{*}-L_{i})'-(u_j-L_{i})'\right|\\
&\le \int_{(c_{k}^{i}(j),d_{k}^{i}(j))}\left|(u_{j}^{*}-L_{i})'\right|+\left|(u_j-L_{i})'\right|.
\end{align*}
 
At this point, note that \begin{align}\label{7deabril}
\int_{(c_{k}^{i}(j),d_{k}^{i}(j))}\left|(u_{j}^{*}-L_{i})'\right|\le \int_{(c_{k}^{i}(j),d_{k}^{i}(j))}\left|(u_{j}-L_{i})'\right|.\end{align} In fact, since $u_j^{*}$ is convex in $(c_{k}^{i},d_{k}^{i})$ we have that $u_j^{*}-L_{i}$ is also convex in that interval, therefore $u_{j}^{*}-L_{i}$ has no local maxima in that interval, considering that $u_{j}^{*}-L_{i}\ge u_{j}-L_{i}$ and that they coincide at the endpoints of the interval we conclude the claim. Now since $\left|(u_{j}-L_{i})'\right|=\left|u_j'-\alpha_i\right|$
we conclude our lemma.
\end{proof}
Now, we need to control the (finitely many) remaining intervals in $D_j$.
\begin{lemma}\label{finitelymany}
We have that $$\int_{\cup_{r=1}^{N+1}(c_{r}(j),d_{r}(j))}\left|(u_j^{*})'-(u^{*})'\right|\to 0.$$
\end{lemma}
\begin{proof}
Assume that there exists, for some $r$, an $\varepsilon_{2}>0$ such that $\int_{(c_{r}(j),d_{r}(j))}\left|(u_j^{*})'-(u^{*})'\right|>\varepsilon_{2}$ for a subsequence of $j$ (that we also index by $j$). Let us take a subsequence such that $c_{r}(j)\to c_r,$ $d_{r}(j)\to d_r$ when $j\to \infty$ (possibly $c_r=-\infty$ or $d_r=+\infty$). Then, Lemma \ref{convexityconvergence} implies that $u^{*}$ is convex in $(c_r,d_r)$ and that $(u_j^{*})'\to (u^{*})'$ a.e in $(c_r,d_r)$. Therefore, by the Brezis-Lieb lemma we just need to prove that $$\underset{j\to \infty}{\lim} \int_{(c_{r}(j),d_{r}(j))}\left|(u_j^{*})'\right|=\int_{(c_r,d_r)}\left|(u^{*})'\right|.$$ If we write $m_r(j)=\underset{x\in (c_r(j),d_r(j))}{\min} u_j^{*}(x)$ and $m_r=\underset{x\in (c_r,d_r)}{\min}u^{*}(x)$, we have that $$\int_{(c_{r}(j),d_{r}(j))}\left|(u_j^{*})'\right|=u_j^{*}(c_r(j))-2m_{r}(j)+u_j^{*}(d_r(j))$$ and $$ \int_{(c_r,d_r)}\left|(u^{*})'\right|=u^{*}(c_r)-2m_r+u^{*}(d_r).$$ Therefore the desired convergence is a consequence of the uniform convergence and the continuity of $u^{*}$. This concludes the proof of the lemma.   
\end{proof}
\section{Proof of Theorem \ref{1}} 
With the tools developed in the last section we are in position to prove our theorem. 
First, we claim the following:
\begin{align}\label{claimproof}
\underset{j\to \infty}{\lim}\int_{D_j\cap C}\left|(u_j^{*})'-(u^{*})'\right|\to 0.
\end{align}

Noticing that $u'=(u^{*})'$ a.e. in the set of integration we have \begin{align*}
\int_{D_j\cap C}\left|(u_j^{*})'-(u^{*})'\right|&=\int_{D_j\cap C}\left|(u_j^{*})'-u'\right|\\
&\le \int_{D_j}\left|(u_j^{*})'-u'\right|\\
&\le \int_{D_{j}^{1}}\left|(u_j^{*})'-u'\right|+\int_{\cup_{i=0}^{N+1}D_{j}^{2,i}}\left|(u_j^{*})'-u'\right|\\
&\le \int_{D_{j}^{1}}\left|(u_j^{*})'-u'\right|+\int_{\cup_{i=0}^{N+1}D_{j}^{2,i}}\left|(u_j^{*})'-u_j'\right|+\|u'-u_j'\|_1\\
&\le  \int_{D_{j}^{1}}\left|(u_j^{*})'-u'\right|+5\varepsilon,  
\end{align*}
for $j$ big enough, where we use Lemma \ref{boundingbiggroups} in the last line. Since $$\int_{D_{j}^{1}}\left|(u_j^{*})'-u'\right|<\varepsilon$$ for $j$ big enough by the Lemma \ref{finitelymany}, we have  $$\underset{j\to \infty}{\limsup} \int_{D_j\cap C}\left|(u_j^{*})'-(u^{*})'\right|\le 6\varepsilon.$$ Since $\varepsilon>0$ is arbitrary we conclude the proof of our claim \eqref{claimproof}.

From \eqref{claimproof} and since $$\int_{C_{j}\cap C}\left|(u_j^{*})'-(u^{*})'\right|=\int_{C_{j}\cap C}\left|u_{j}'-u'\right|\to 0,$$ we conclude that 
\begin{align}\label{convergenceinC}
\int_{C}\left|(u_j^{*})'-(u^{*})'\right|\to 0.   
\end{align}
Now, in order to prove our Theorem \ref{1} we need to conclude that $$\int_{D}\left|(u_j^{*})'-(u^{*})'\right|\to 0.$$ Indeed, in light of Lemma \ref{pointwisederivative}, by the Brezis-Lieb lemma we only require that $$\int_{D}\left|(u_j^{*})'\right|\to \int_{D}\left|(u^{*})'\right|,$$ and this is a consequence of \eqref{convergenceinC} and Proposition \ref{grk}. This concludes the proof of Theorem \ref{1}.
\subsection{Concluding remarks}
The same scheme of proof presented here allows one to establish the analogous of Theorem \ref{1} for a more general class of maximal operators of convolution type. The key properties that we require are that the maximal function $u^{*}$ is continuous and has the {\it subharmonicity property}, and one has to then deal with minor technicalities that might appear (and for simplicity we do not enter in all such variations). For instance, one could consider the operators defined in \cite[Section 1.2]{CFS2015}, in which the approximation of the identity are slightly different than \eqref{approximationofidentity}.  
\section{Acknowlegdements}
The author is grateful to E. Carneiro for the encouragement and very helpful discussions. The author is also thankful to M. Sousa for helpful comments about earlier versions of this manuscript.  
\bibliography{Reference}
\bibliographystyle{amsplain} 
\end{document}